\documentclass[12pt]{amsart}

\usepackage[pdfauthor   = {Mohamed\ Barakat},
            pdftitle    = {On\ subdirect\ factors\ of\ projective\ modules\ and\ applications\ to\ system\ theory},
            pdfsubject  = {13C10; 13H10; 18E10; 13P10; 13P20; 18G05; 18G15; 13P25; 93B05; 93B25; 93B40},
            pdfkeywords = {subdirect\ products;\ torsionless;\ grade;\ projective;\ torsion-free;\ codimension;\ free;\ constructive\ homological\ algebra;\ Abelian\ category;\ homalg;\ system\ theory},
            bookmarks=true,
            bookmarksopen=true,
            pagebackref=true,
            hyperindex=true,
            colorlinks=true,
            linkcolor=blue,
            citecolor=blue,
            filecolor=blue,
            urlcolor=blue,
            ]{hyperref}

\usepackage[utf8,utf8x]{inputenc}
\usepackage[english]{babel}
\usepackage[T1]{fontenc}
\usepackage{geometry}                
\usepackage{times}
\usepackage{mathrsfs}
\usepackage{latexsym}
\usepackage{amssymb}
\usepackage{amsthm}
\usepackage{epsfig}
\usepackage{colortbl}
\usepackage[all]{xy}
\usepackage{listings}
\usepackage{fancyvrb}
\usepackage{graphicx}
\usepackage[dvipsnames]{xcolor}
\usepackage{accents} 
\usepackage{enumerate}
\usepackage{wrapfig}
\usepackage{tikz}
\usetikzlibrary{shapes,arrows,matrix,backgrounds,positioning,plotmarks,calc,patterns,matrix,decorations.pathreplacing,decorations.pathmorphing,decorations.text}
\usepackage[colorinlistoftodos,shadow]{todonotes}

\usepackage{multirow}
\usepackage{mdwlist}
\usepackage{stmaryrd}
\usepackage{mathdots} 

\usepackage[toc,page]{appendix}
\usepackage{float}

\newtheoremstyle{mytheoremstyle} 
    {5pt}                    
    {5pt}                    
    {\itshape}                   
    {\parindent}                           
    {\bf}                   
    {.}                          
    {.5em}                       
    {}  

\theoremstyle{mytheoremstyle}

\newtheorem{theorem}{Theorem}[section]

\newtheorem{lemm}[theorem]{Lemma}
\newtheorem{prop}[theorem]{Proposition}
\newtheorem{coro}[theorem]{Corollary}

\newtheoremstyle{mytdefintionstyle} 
    {5pt}                    
    {5pt}                    
    {\rm}                   
    {\parindent}                           
    {\bf}                   
    {.}                          
    {.5em}                       
    {}  

\theoremstyle{remark}
\newtheorem{rmrk}[theorem]{Remark}

\theoremstyle{mytdefintionstyle}
\newtheorem{defn}[theorem]{Definition}

\newtheoremstyle{exmp_contd} 
{\topsep} {\topsep}%
{\upshape}
{}
{\bfseries}
{}
{ }
{\thmname{#1}\,\thmnumber{ #2}\thmnote{#3}\enspace(continued)}

\theoremstyle{exmp_contd}

\DeclareMathOperator{\Ann}{Ann}
\DeclareMathOperator{\Ass}{Ass}
\DeclareMathOperator{\codim}{codim}

\DeclareMathOperator{\depth}{depth}
\DeclareMathOperator{\Ext}{Ext}
\DeclareMathOperator{\grade}{grade}
\DeclareMathOperator{\Hom}{Hom}

\DeclareMathOperator{\Spec}{Spec}

\DeclareMathOperator{\tor}{t}

\newcommand\A{\mathcal{A}}

\newcommand\F{\mathcal{F}}
\newcommand{\Z}{\mathbb{Z}}
\newcommand{\Q}{\mathbb{Q}}
\newcommand{\R}{\mathbb{R}}
\newcommand{\C}{\mathbb{C}}

\renewcommand{\phi}{\varphi}
\renewcommand{\epsilon}{\varepsilon}

\definecolor{darkgreen}{rgb}{0.008,0.617,0.067}
\definecolor{brown}{rgb}{0.6,0.4,0.2}

\begin{document}

\lstset{basicstyle=\scriptsize\ttfamily,
        frame=single}
        
\author{Mohamed Barakat}
\address{Department of mathematics, University of Kaiserslautern, 67653 Kaiserslautern, Germany}
\email{\href{mailto:Mohamed Barakat <barakat@mathematik.uni-kl.de>}{barakat@mathematik.uni-kl.de}}

\title[On subdirect factors of a projective module and applications to system theory]{On subdirect factors of a projective module \\ and applications to system theory}


\begin{abstract}
  We extend a result of \textsc{Napp Avelli}, \textsc{van der Put}, and \textsc{Rocha} with a system-theoretic interpretation to the noncommutative case:
  Let $P$ be a f.g.~projective module over a two-sided \textsc{Noether}ian domain.
  If $P$ admits a subdirect product structure of the form $P \cong M \times_T N$ over a factor module $T$ of grade at least $2$ then the torsion-free factor of $M$ (resp.\  $N$) is projective.
\end{abstract}


\maketitle


\section{Introduction}

This paper provides two homologically motivated generalizations of a module-theoretic result proved by \textsc{Napp Avelli}, \textsc{van der Put}, and \textsc{Rocha}. This result was expressed in \cite{NR10} using the dual system-theoretic language and applied to behavioral control.
Their algebraic proof covers at least the polynomial ring $R=k[x_1,\ldots,x_n]$ and the \textsc{Laurent} polynomial ring $R=k[x_1^\pm,\ldots,x_n^\pm]$ over a field $k$.
The corresponding module-theoretic statement is the following:

\begin{theorem} \label{coro:diego}
  Let $R$ be one of the above rings and $M = R^q / A$ a finitely generated \emph{torsion-free} module.
  If there exists a submodule $B \leq R^q$ such that $A \cap B = 0$ and $T := R^q/(A+B)$ is of codimension at least $2$ then $M$ is free.
\end{theorem}

In fact, they prove a more general statement of which the previous is obviously a special case.
However, the special statement implies the more general one.

\begin{theorem}[{\cite[Theorem~18]{NR10}}] \label{thm:diego}
  Let $R$ be one of the above rings and $M = R^q / A$ a finitely generated module.
  If there exists a submodule $B \leq R^q$ such that $A \cap B = 0$ and $T := R^q/(A+B)$ is of codimension at least $2$ then the torsion-free factor $M/ \tor(M)$ of $M$ is free.
\end{theorem}

In the proposed module-theoretic generalization of Theorem~\ref{coro:diego} the notions ``torsion-free'', ``codimension'' and ``free'' are replaced by the more homological notions ``torsionless'', ``grade'', and ``projective'', respectively.

We start by describing the very basics of the duality between linear systems and modules in Section~\ref{sec:duality}.
The two notions ``torsionless'' and ``grade'' are briefly recalled in Section~\ref{sec:torsionless_grade}.
In Section~\ref{sec:module-theoretic_generalization1} a module-theoretic generalization of Theorem~\ref{coro:diego} is stated and proved.
The proof relies on an \textsc{Abel}ian generalization which is treated in Section~\ref{sec:Abelian_generalization}.
Since torsion-freeness admits a system-theoretic interpretation we need to discuss the relation between being torsion-free and being torsionless to justify the word ``generalization''.
Indeed, torsionless modules are torsion-free but the converse is generally false (cf.~Remark~\ref{rmrk:Q_not_torsionless} for a precise statement).
Section~\ref{sec:torsionless=torsion-free} describes a fairly general setup in which the converse does hold.
And only when it holds are we able to prove the corresponding generalization of Theorem~\ref{thm:diego}.
This is done in Section~\ref{sec:module-theoretic_generalization2}.
Finally, Appendix~\ref{sec:converse_key_lemma} contains a converse to the key Lemma of this paper.

\medskip
\textbf{Convention:} Unless stated otherwise $R$ will always denote a not necessarily commutative unitial ring.
The term ``domain'' will not imply commutativity.

Everything below is valid for left and for right $R$-modules.

\section{Duality between linear system theory and module theory} \label{sec:duality}

For an $R$-module $\F$ we define the category of \textbf{$\F$-behaviors} as the image of the contravariant $\Hom$-functor $\Hom_R(-,\F): R\mathrm{-Mod} \to \mathrm{Mod-}C$, where $C$ is the center of $R$ (or the endomorphism ring of $\F$ or any unitial subring of thereof).
$R$ is called the ring of functional operators and $\F$ a signal module or signal space.

An $R$-module $M$ is said to be cogenerated by $\F$ if $M$ can be embedded into a direct power $\F^I$ for some\footnote{One can take the index set to be the solution space $I := \Hom_R(M,\F)$ and require in the definition that the evaluation map from $M$ to the direct power $\F^I$, sending $m \in M$ to the map $I \to \F, \phi \mapsto \phi(m)$, is an embedding.} index set $I$.
$\F$ is a called a cogenerator if it cogenerates any $R$-module $M$, or, equivalently, if the duality functor $\Hom_R(-,\F)$ is faithful.
In particular, a cogenerator is a faithful $R$-module.
The duality functor $\Hom_R(-,\F)$ is exact if and only if $\F$ is injective.
An injective $\F$ is a cogenerator if and only if the solution space $\Hom_R(M,\F) \neq 0$ for each\footnote{Cyclic or even simple $R$-modules suffice (cf.~\cite[Theorem~3.1]{TI64}).} $M \neq 0$.
In particular, all simples can be embedded into an injective cogenerator.
Summing up, $\Hom_R(-,\F)$ is exact and reflects exactness\footnote{This follows easily from the fact that an exact faithful functor of \textsc{Abel}ian categories is conservative, i.e., reflects isomorphisms (see, e.g., \cite[Lemma~A.1]{BL_SerreQuotients}).
An exact functor of \textsc{Abel}ian categories which also reflects exactness is called ``faithfully exact'' in  \cite[Definition~1]{TI64}.
Injective cogenerators are called ``faithfully injective'' in \cite[Definition~3]{TI64}.} (and hence faithful) if and only if $\F$ is an injective cogenerator.
In this case the $\Hom$-duality between $R$-modules and $\F$-behaviors is perfect.
The above statements are true in any \textsc{Abel}ian category with products \cite[§~IV.6]{Sten75}.

The \textsc{Abel}ian group $\Q/\Z$ of characters of $\Z$ is an injective cogenerator in the category of \textsc{Abel}ian groups.
Likewise, the $R$-module $\Hom_\Z(R,\Q/\Z)$ is called the module of characters of $R$ and is an injective cogenerator in $R\mathrm{-Mod}$.
This follows from the adjunction between $\Hom$ and the tensor product functor \cite[Proposition~I.9.3]{Sten75}.
The $k$-dual $R^\vee := \Hom_k(R,k)$ is an injective cogenerator for each $k$-algebra $R$ over a field $k$.
This classical result was already used in \cite[Corollary 3.12, Remark 3.13]{OB}.
\textsc{Plesken} and \textsc{Robertz} gave a constructive proof for the injectivity of the $k$-dual $R^\vee$ when $R$ is a multiple \textsc{Ore} extension over a computable field $k$ admitting a \textsc{Janet} basis notion  (cf.~\cite[Corollary 4.3.7, Theorem 4.4.7]{robphd}).
Furthermore, a minimal injective cogenerator always exists \cite[Proposition~19.13]{Lam} (see \cite[Subsection~19A]{Lam} for more details on injective cogenerators).

However, only those injective cogenerators which can be interpreted as a space of ``generalized functions'' (like distributions, hyperfunctions, microfunctions) are of direct significance for system theory in the engineering sense.
\textsc{Oberst} considers in \cite{OB} injective cogenerators $\F$ over commutative \textsc{Noether}ian rings which are large, i.e., satisfying $\Ass(\F)=\Spec(R)$.
\textsc{Fröhler} and \textsc{Oberst} prove in \cite{FO98} that the space of \textsc{Sato} hyperfunctions on an open interval $\Omega \subset \R$ is an injective cogenerator for the noncommutative ring $R:=A\left[\frac{\mathrm{d}}{\mathrm{d}t}\right]$ where $A := \left\{\frac{f}{g} \mid f,g \in \C[t], \forall \lambda \in \Omega: g(\lambda) \neq 0 \right\}$.
\textsc{Zerz} shows in \cite{Zer06} that the space of $\R$-valued functions on $\R$ which are smooth except at finitely many points\footnote{As a referee remarked, this restriction rules out singularities which solutions of ODEs with varying coefficients might generally exhibit.} is an injective cogenerator for the rational \textsc{Weyl} algebra $B_1(\R)=\R(t)[\frac{\mathrm{d}}{\mathrm{d}t}]$.

From now on let $\F$ be an injective cogenerator with system-theoretic relevance.
Restricting to factor modules $M=R^q/A$ of a fixed free module $R^q$ yields a (non-intrinsic) \textsc{Galois} duality between the submodules $A$ of $R^q$, the so-called equations submodules, and $\F$-behaviors $\mathscr{M} = \Hom_R(R^q/A, \F)$.

In system-theoretic terms a factor module of the module $M$ corresponds to a subbehavior of $\mathscr{M} = \Hom_R(M,\F)$, and the torsion-free factor to the largest controllable subbehavior.
All degrees of torsion-freeness (including reflexivity and projectivity) are related to successive parametrizability of multi-dimensional systems in \cite{PQ99,CQR05}.
Freeness of modules corresponds to flatness of linear systems \cite{Fl}.
A common factor module $T$ of $M=R^q / A$ and $N=R^q / B$ corresponds to the so-called interconnection, i.e., the intersection $\mathscr{M} \cap \mathscr{N}$ of the two behaviors $\mathscr{M},\mathscr{N}$ corresponding to $M,N$.
The interconnection is called regular when $A \cap B = 0$.
Finally, the codimension of a module corresponds to the degree of autonomy of the corresponding behavior.
This paper suggests, in particular, the use of grade as a substitute for codimension to define the degree of autonomy in the noncommutative setting.

\section{Torsionless modules and grade} \label{sec:torsionless_grade}

We will use the notion of a torsionless module, due to \textsc{H.~Bass}, to provide a natural module-theoretic generalization of Theorem~\ref{coro:diego}.

\begin{defn}
  An $R$-module $M$ is called \textbf{torsionless} if it is cogenerated by the free module $R$, i.e., if it can be embedded into a direct power $R^I := \prod_{i \in I} R$, for some index set $I$.
\end{defn}

\begin{rmrk} \label{rmrk:torsionless}
From the definition we conclude that:
  \begin{enumerate}
    \item Any submodule of a torsionless module is torsionless and any direct product (and hence sum) of torsionless modules is torsionless.
    \item Since direct sums embed in direct products any submodule of a free module is torsionless.
      Thus, projective modules and left and right ideals are torsionless. \label{rmrk:torsionless.projective}
  \end{enumerate}
\end{rmrk}

We denote by $M^*:=\Hom_R(M,R)$ the $R$-dual of an $R$-module $M$.
It is easy to see that $M$ is torsionless iff\footnote{
The ``only if''-part follows by setting $\lambda$ to be the composition of the embedding $\jmath:M \hookrightarrow R^I$ and the projection $\pi_\iota: R^I \to R$ such that $\pi_\iota(\jmath(m)) \neq 0$.
The ``if''-part follows by setting $I=M^*$ and $\jmath$ to be the evaluation map $\epsilon_M: M \to M^{**}, m \mapsto (\lambda \mapsto \lambda(m))$  considered as a map to $R^I \supset M^{**}$.} for any $m \in M \setminus \{0\}$ there exists a functional $\lambda \in M^*$ such that $\lambda(m) \neq 0$.
Hence, $M$ is torsionless iff the natural evaluation map
\[
  \epsilon_M: M \to M^{**}, m \mapsto (\lambda \mapsto \lambda(m))
\]
is a monomorphism\footnote{Recall, $M$ is called  \textbf{reflexive} if $\epsilon_M$ is an isomorphism.}.
The dualized evaluation map $\epsilon_M^*:M^{***} \to M^*$ is a post-inverse of the evaluation map of the dual module $\epsilon_{M^*}:M^* \to M^{***}$, i.e., the latter is a split monomorphism (cf.~\cite[Remark~(4.65).(f)]{Lam99}).
In particular, the dual $M^*$ and the double-dual $M^{**}=(M^*)^*$ are torsionless modules.
This gives rise to the following definition.

\begin{defn}
  The \textbf{torsionless factor} of an $R$-module $M$ is the coimage $M/\ker \epsilon_M$ of the evaluation map.
\end{defn}

\begin{rmrk} \label{rmrk:Q_not_torsionless}
  If $R$ is a domain then any torsionless module is torsion-free.
  The converse is false: The infinitely generated $\Z$-module $\Q$ is torsion-free with a zero evaluation map, i.e., the ``opposite'' of being torsionless.
  Finitely generated modules behave better in this respect (cf.~Theorem~\ref{thm:torsionless=torsion-free}).
  While submodules of torsion modules are torsion, the torsionless $\Z$-submodule $\Z \leq \Q$ shows that having a zero evaluation map is not stable under passing to submodules.
  Still, the factor module $\Q/\Z$ has a zero evaluation map.
\end{rmrk}

Recall, an $R$-module $T$ is is said to have \textbf{grade} at least\footnote{This is a more convenient than defining the grade by an equality.} $c$ if $\Ext^i(T,R)=0$ for all $i<c$.
The grade of the associated cyclic module $R/\Ann(T)$ coincides with the grade of $T$.

\begin{rmrk} \label{rmrk:CM}
  Let $R$ be a commutative \textsc{Noether}ian ring. 
  The grade of an $R$-module $T$ coincides, by a theorem of \textsc{Rees}, with $\depth \Ann(T) := \depth(\Ann(T),R)$ \cite[Proposition~18.4]{eis}, \cite[Theorem~1.2.5]{BrunsHerzog}.
  $R$ is called \textsc{Cohen-Macaulay} if the notions of codimension and grade coincide, i.e., if $\codim T := \codim \Ann(T)$ coincides with $\grade T = \depth \Ann(T)$ for all modules $T$ \cite[Introductions to Chapters~9 and~18]{eis}.
  The reader is referred to \cite[Part~II]{BrunsHerzog} for large classes of \textsc{Cohen-Macaulay} rings.
\end{rmrk}

The following definition is used to formulate the module-theoretic generalization of Theorem~\ref{coro:diego}.

\begin{defn} \label{defn:projective_up_to_grade}
  We say that an $R$-module $M$ is \textbf{projective up to grade $c$} if there exists a projective module $P$ and an epimorphism $P \stackrel{\pi}{\twoheadrightarrow} M$ such that $A := \ker \pi \leq P$ admits \textbf{a complement up to grade $c$ in $P$}, i.e., if there exists a submodule $B \leq P$ with $A \cap B=0$ and $T:=P / (A + B)$ has grade at least $c$.
  If $M$ is finitely generated then we insist that $P$ is finitely generated.
\end{defn}

\section{A module-theoretic generalization of Theorem~\ref{coro:diego}} \label{sec:module-theoretic_generalization1}

Projective modules are torsionless (Remark~\ref{rmrk:torsionless}.\eqref{rmrk:torsionless.projective}) and obviously projective up to grade $c$, for any $c$.
The converse is true for finitely generated modules and $c \geq 2$, yielding a module-theoretic generalization of Theorem~\ref{coro:diego}:
\begin{theorem} \label{coro:main1}
  Let $R$ be a ring and $M$ a finitely generated $R$-module.
  If $M$ is torsionless and projective up to grade $2$ then $M$ is projective.
\end{theorem}
\begin{proof}
  Let $P \stackrel{\pi}{\twoheadrightarrow} M$ be the f.g.~projective module of Definition~\ref{defn:projective_up_to_grade}, $A := \ker \pi \leq P$, $B$ be a complement up to the grade $2$ of $A$ in $P$, and $T := P / (A + B)$ the factor module of grade $\geq 2$, i.e., $\Hom(T,R)=0=\Ext^1(T,R)$.
  The assertion will follow from Theorem~\ref{thm:main1} as soon as we have shown that $\Hom(T,M) = 0 = \Ext^1(T,P)$ which we will do now: \\
  Since $M$ is torsionless there exists an embedding\footnote{Proposition~\ref{prop:embed_in_free} provides an alternative embedding into a free module of finite rank.} $\jmath: M \hookrightarrow R^I$ in a direct product for some index set $I$.
  As the left exact covariant $\Hom$-functor commutes with direct products \cite[Proposition~I.3.5]{HS} it follows that
  \[
    \Hom(T,M) \cong \Hom(T, \jmath(M)) \leq \Hom(T,R^I) \cong \Hom(T,R)^I=0 \mbox{.}
  \]
  And since $P$ is finitely generated projective it is a direct summand of a free module $R^p \cong P \oplus P'$ of finite rank $p$.
  Finally, the additivity of $\Ext^1(T,-)$ yields
  \[
    \Ext^1(T,P) \leq \Ext^1(T,P) \oplus \Ext^1(T,P') \cong \Ext^1(T,R^p) = \Ext^1(T,R)^p = 0 \mbox{.} \qedhere
  \]
\end{proof}

\section{An \textsc{Abel}ian generalization of Theorem~\ref{coro:diego}} \label{sec:Abelian_generalization}

\begin{wrapfigure}[6]{r}{2.7cm}
\vskip -0.8cm
\begin{tikzpicture}
  \coordinate (dl) at (-1,-1);
  \coordinate (dr) at (1,-1);
  
  \node (T) {$T$};
  \node (M) at (dl) {$M$};
  \node (N) at (dr) {$N$};
  \node (P) at ($(M)+(dr)$) {$P$};
  \node (A) at ($(P)+(dr)$) {$A$};
  \node (B) at ($(P)+(dl)$) {$B$};
  
  \draw [->>] (M) -- (T);
  \draw [->>] (N) -- (T);
  \draw [->>] (P) -- (M);
  \draw [->>] (P) -- (N);
  \draw [left hook->] (A) -- (P);
  \draw [right hook->] (B) -- (P);
  
\end{tikzpicture}
\end{wrapfigure}
Let $\A$ be an \textsc{Abelian} category and $P \cong M \times_T N \in \A$ a subdirect product\footnote{Also called fiber product.}  of two objects $M$ and $N$ over a common factor object $T$, i.e., $M \twoheadleftarrow P \twoheadrightarrow N$ is the pullback of the two epis $M \twoheadrightarrow T \twoheadleftarrow N$.
\begin{theorem} \label{thm:main1}
  If $\Hom(T,M) = 0 = \Ext^1(T,P)$ then the epi $P \twoheadrightarrow M$ is split and $M$ is isomorphic to a direct summand of $P$.
  If furthermore $P$ is projective then so is $M$.
\end{theorem}

The following simple lemma is the essence of the short proof of Theorem~\ref{thm:main1}.
We keep the above notation and set $A := \ker \left( P \twoheadrightarrow M \right)$ and $B := \ker\left( P \twoheadrightarrow N \right)$.

\begin{lemm} \label{lemm:key}
  If $\Ext^1(T,A) = 0$ then $A$ has a complement $B'\cong M$ in $P$ which contains $B$.
  In particular, $M$ is isomorphic to a direct summand of $P$.
\end{lemm}

\vskip -0.4cm
\begin{wrapfigure}[8]{r}{4.7cm}
\vskip -0.55cm
\begin{tikzpicture}[C/.style={color=red,line width=1.5pt,dotted},new/.style={color=blue,line width=2pt},map/.style={color=gray,dashed}]
  \coordinate (d) at (0,-1);
  \coordinate (dl) at (-1.5,-1.5);
  \coordinate (l) at (-1,0);
  \coordinate (dr) at (1,-1);
  
  \coordinate (P);
  \coordinate (B') at (dr);
  \coordinate (S) at (d);
  \coordinate (A) at ($(S)+(dl)$);
  \coordinate (B) at ($(S)+(dr)$);
  \coordinate (0) at ($(A)+(dr)$);
  
  \fill (P) circle (2pt);
  \fill (B') circle (2pt);
  \fill (S) circle (2pt);
  \fill (A) circle (2pt);
  \fill (B) circle (2pt);
  \fill (0) circle (2pt);
  
  \draw (0) -- (A) -- (S) -- (P);
  \draw (0) -- (B) -- (S);
  \draw (B) -- (B') -- (P);
  \draw[decorate,decoration=brace] ($(S)+0.1*(l)-0.1*(d)$) -- node [left] {\small $T$} ($(P)+0.1*(l)+0.1*(d)$);
  \draw[decorate,decoration=brace] ($(A)+0.2*(l)$) -- node [left] {\small $M$} ($(P)+1.7*(l)$);
  \draw[decorate,decoration=brace] ($(P)-1.7*(l)$) -- node [right] {\small $N$} ($(B)-0.7*(l)$);
  \draw[dotted,gray,thick] ($(P)+1.5*(l)$) -- ($(P)+0.22*(l)$);

  \draw[dotted,gray,thick] ($(P)-1.5*(l)$) -- ($(P)-0.22*(l)$);
  \draw[dotted,gray,thick] ($(B)-0.5*(l)$) -- ($(B)-0.22*(l)$);
  
  \node at (P) [above right=-0.1em] {\small $P$};
  \node at (B') [right] {\small $B'$};
  \node at (S) [left=0.4em] {\small $S$};
  \node at (A) [below left] {\small $A$};
  \node at (B) [below right] {\small $B$};
  
\end{tikzpicture}
\end{wrapfigure}
\mbox{}
\begin{proof}
  Set $S:=A+B \leq P$, the direct sum of $A$ and $B$.
  The assumption $\Ext^1(T,A) = 0$ and the natural isomorphism $S/B \cong A$ imply that the short exact sequence $0 \to S/B \to N \to T \to 0$ splits.
  In other words, there exists a subobject $B'$ of $P$ with $B' \geq B$ such that $B'/B$ is a complement of $S/B \cong A$ in $P/B$.
  Since $B' \cap (A+B) = B$ it follow that $B'$ is a complement of $A$ in $P$, canonically isomorphic to $M$.
\end{proof}

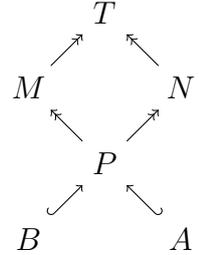
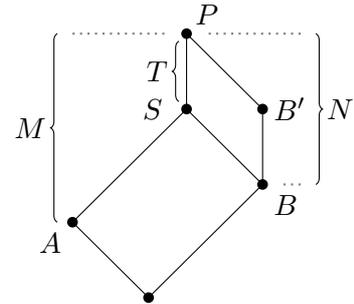
\begin{proof}[Proof of Theorem~\ref{thm:main1}]
  To apply Lemma~\ref{lemm:key} we need to show that $\Ext^1(T,A) = 0$.
  Indeed, $\underbrace{\Hom(T,M)}_{=0} \to \Ext^1(T,A) \to \underbrace{\Ext^1(T,P)}_{=0}$ is part of the long exact $\Ext(T,-)$-sequence with respect to the short exact sequence $0 \to A \to P \to M \to 0$.
  Hence, $\Ext^1(T,A) = 0$.
\end{proof}

Lemma~\ref{lemm:key} has an interesting converse which we did not need here.
It is treated in Appendix~\ref{sec:converse_key_lemma}.

\section{When does torsion-free imply torsionless?} \label{sec:torsionless=torsion-free}

In this section we assume $R$ to be two-sided \textsc{Noether}ian\footnote{$R$ two-sided coherent is, as usual, enough but we stick to two-sided \textsc{Noether}ian for lack of references.}.
A \textbf{finite projective presentation} of a f.g.~$R$-module $M$ is an exact sequence
 $M \twoheadleftarrow P_0 \xleftarrow{\partial} P_1$ with f.g.~projective $R$-modules $P_0$ and $P_1$.
For such a module define the \textbf{\textsc{Auslander} dual} $A(M)$ to be the cokernel of the dual (or pullback) map $\partial^*:P_0^* \to P_1^*$.
Like the syzygy modules of $M$, the \textsc{Auslander} dual is well-defined up to projective equivalence.
In particular $\Ext^i(A(M),R)$ does not depend on the finite projective presentation for $i>0$.
Furthermore, if $M$ is projective then $A(M)=0$ (up to projective equivalence) and $\Ext^i(A(M),R)=0$ for all $i>0$ (for a converse statement cf.~\cite[Theorem~7]{CQR05}).

The kernel and cokernel of the evaluation map $\epsilon:M \to M^{**}$ were characterized by \textsc{Auslander}, where $M$ is assumed to have a finite projective presentation.
As one of many applications of his theory of coherent functors \cite{A} he proved the existence of a natural monomorphism $\tau: \Ext^1(A(M),R) \hookrightarrow M$ and a natural epimorphism $\rho: M^{**} \twoheadrightarrow \Ext^2(A(M),R)$ such that
\begin{equation} \label{epsilon} \tag{$\epsilon$}
  0 \to \Ext^1(A(M),R) \xrightarrow{\tau} M \xrightarrow{\epsilon} M^{**} \xrightarrow{\rho} \Ext^2(A(M),R) \to 0
\end{equation}
is an exact sequence.
In particular, $M$ is torsionless iff $\Ext^1(A(M),R)=0$ and reflexive iff $\Ext^i(A(M),R)=0$ for $i=1,2$.
A short elegant proof of \eqref{epsilon} can be found in \cite[Theorem~6]{CQR05} and a generalization in  \cite[Chapter~2, (2.1)]{AB} (see also \cite[Exer.~IV.7.3]{HS}).

A left (resp.~right) \textsc{Noetherian} domain $R$ satisfies the left (resp.~right) \textsc{Ore} condition and the set of torsion elements $\tor(M)$ of an $R$-module $M$ form an $R$-submodule.
The following theorem states that the two notions ``torsion-free'' and ``torsionless'' coincide for finitely generated modules.

\begin{theorem}[{\cite[Theorem~5]{CQR05}}] \label{thm:torsionless=torsion-free}
  Let $R$ be a two-sided \textsc{Noether}ian domain and $M$ a f.g.~$R$-module.
  Then the image of the natural monomorphism $\tau:\Ext^1(A(M),R) \to M$ is the torsion submodule $\tor(M)$ yielding a canonical isomorphism $\Ext^1(A(M),R) \cong \tor(M)$.
  In particular, the torsion-free factor and the torsionless factor of $M$ coincide and $M$ is torsion-free iff $M$ is torsionless.
\end{theorem}

\begin{coro}
  Let $R$ be a two-sided \textsc{Noether}ian domain.
  A finitely generated $R$-module of grade at least $1$ is torsion.
\end{coro}
\begin{proof}
  Let $T$ be a such a module.
  By Theorem~\ref{thm:torsionless=torsion-free} the torsion-free factor coincides with the torsionless factor.
  The latter is trivial since $\Hom_R(T,R)=0$ and the evaluation map $T \to T^{**}$ vanishes.
  Hence $T$ is a torsion module.
\end{proof}

Any finitely generated torsion-free module over a commutative domain can be embedded into a free module of finite rank.
This can be easily seen by passing to the quotient field (cf.~\cite[the paragraph preceding (2.31)]{Lam99}).
The exact sequence \eqref{epsilon} yields a generalization to the noncommutative case.
The following proposition is part of \cite[Theorem~8]{CQR05}.

\begin{prop} \label{prop:embed_in_free}
  Let $R$ be a two-sided \textsc{Noether}ian domain.
  A finitely generated torsionless (=torsion-free) $R$-module can be embedded in a free module of finite rank.
\end{prop}
\begin{proof}
  The two-sided coherence of $R$ assures the existence of finite rank free resolutions for f.g.~$R$-modules.
  Let $M \twoheadleftarrow F_0 \xleftarrow{\partial} F_1$ be a finite free presentation of $M$, i.e., with free modules $F_0$ and $F_1$ of finite rank.
  Dualizing we obtain a finite free presentation of $A(M)$ which we can resolve one step further and obtain $F_{-1}^* \to F_0^* \xrightarrow{\partial^*} F_1^* \to A(M)$ with $F_{-1}^*$ free of finite rank.
  Dualizing again yields an exact complex: The defect of exactness at $F_0^{**}$ is $\Ext^1(A(M),R)$ which vanishes since $M$ is torsionless.
  Using the reflexiveness of free modules of finite rank it follows that $M$ (as the cokernel of $\partial$ or $\partial^{**}$) embeds into the finite rank free module $F_{-1}^{**}$.
\end{proof}

The above Proposition is implemented for computable rings in \textsc{OreModules} \cite{CQR07} and $\mathtt{homalg}$ \cite{homalg-project,BL}.

\section{A module-theoretic generalization of Theorem~\ref{thm:diego}} \label{sec:module-theoretic_generalization2}

\begin{theorem}
  Let $R$ be a two-sided \textsc{Noether}ian domain and $M$ a finitely generated $R$-module.
  If $M$ is projective up to grade $2$ then the torsion-free factor $M/ \tor(M)$ is projective.
\end{theorem}
\begin{proof}
  Let  $P \stackrel{\pi}{\twoheadrightarrow} M$ be the f.g.~projective module of Definition~\ref{defn:projective_up_to_grade}, $A := \ker \pi \leq P$, $B$ be a complement up to the grade $2$ of $A$ in $P$, and $T := P / (A + B)$ the factor module of grade $\geq 2$.
  Let $A'$ denote the preimage of $\tor(M)$ in $P$, so $A'/A \cong \tor(M)$.
  The intersection $A' \cap B = 0$ since $A' \cap (A+B) = A$.
  The latter can be seen as follows:
  Otherwise $(A' \cap (A+B)) / A \leq A'/A \cong \tor(M)$ would be a nontrivial torsion\footnote{Here we need that $\tor(M)$ is torsion and not merely having a zero evaluation map.} submodule of the torsion-free factor $(A+B) / A \cong B$.
  The next proposition guarantees that the epimorphic image $T''=P / (A' + B)$ of $T$ is again of grade at least $2$.
  It remains to apply Theorem~\ref{coro:main1} to $M / \tor(M) \cong P / A'$ with $B$ now a complement of $A'$ in $P$ up to grade at least $2$.
\end{proof}

\begin{prop}
  Let $T$ be a torsion module  over a domain. If $T$ has grade at least $2$ then any of its factor modules has grade at least $2$.
\end{prop}
\begin{proof}
  The grade condition for $T$ means that $\Hom(T,R)=0=\Ext^1(T,R)$.
  Let $T'' = T/T'$ be a factor of $T$.
  Any morphism from a torsion module over a domain into a torsion-free module is zero.
  An since the submodule $T'$ is again torsion\footnote{Here we need that $T$ is torsion and not merely having a zero evaluation map.} it follows that $\Hom(T',R)=0$.
  The long exact $\Ext(-,R)$-sequence (w.r.t.~$0 \to T' \to T \to T'' \to 0$)
  \[
    0 \to \Hom(T'',R) \to \underbrace{\Hom(T,R)}_{=0} \to \underbrace{\Hom(T',R)}_{=0} \to \Ext^1(T'',R) \to \underbrace{\Ext^1(T,R)}_{=0}
  \]
  implies that $\Hom(T'',R) = 0 = \Ext^1(T'',R)$.
\end{proof}

We end this section by describing a context in which the original formulation can be retained.
If $M$ has a finite free resolution, e.g., if $R$ is an FFR ring\footnote{Finite free resolution ring.}, then, by a remark of \textsc{Serre}, $M$ projective implies $M$ stably free (cf.~\cite[Proposition~19.16]{eis}).
If, additionally, $R$ is \textsc{Hermite} then $M$ projective already implies $M$ free.
If $R$ is commutative \textsc{Cohen-Macaulay} then the notions of grade and codimension coincide (cf.~Remark~\ref{rmrk:CM}).
The rings mentioned in the Introduction are FFR, \textsc{Hermite}, and commutative \textsc{Cohen-Macaulay} (even regular) domains.

\begin{rmrk}
It should be noted that this paper is less of computational interest as non of the results suggests an algorithm to decide the projectivity of the torsion-free factor of a given finitely presented module.
For an overview on algorithms to test projectivity, stably freeness, and freeness see \cite[Subsection~3.4]{BL} and the references therein.
However, given $M=R^q/A$ and $B \leq R^q$ over a computable ring $R$ it can be algorithmically decided whether $A \cap B = 0$ and $\grade T \geq 2$ for $T=R^q/(A+B)$.
For the definition of a computable ring see \cite[Definition~3.2]{BL}.
The torsion-free factor over finitely presented modules over such rings can be computed, e.g., as the coimage of the evaluation map.
\end{rmrk}

\appendix

\section{A converse of Lemma~\ref{lemm:key}} \label{sec:converse_key_lemma}

\def\objZZ{B}
\def\objYY{S}
\def\objYZ{A}
\def\objXX{P}
\def\objXZ{N}
\def\objXY{T}

\begin{wrapfigure}[8]{r}{4cm}
\vskip -0.9cm
\begin{minipage}{4cm}
\[
  \xymatrix@=0.2cm{
  &
  &
  0
  \ar[dr]
  &
  &
  0
  \ar[dl]
  \\
  &
  0
  \ar[dr]
  &
  &
  \objZZ
  \ar@/^1.4pc/[dd]
  \ar[dl]
  \\
  0
  \ar[dr]
  &
  &
  \objYY
  \ar[dl]
  \ar[dr]
  \\
  &
  \objYZ
  \ar[dl]
  \ar[dr]
  &
  &
  \objXX
  \ar@/^1.4pc/[dd]
  \ar[dl]
  \\
  0
  &
  &
  \objXZ
  \ar[dl]
  \ar[dr]
  \\
  &
  0
  &
  &
  \objXY
  \ar[dl]
  \ar[dr]
  \\
  &
  &
  0
  &
  &
  0
  }
\]
\end{minipage}
\end{wrapfigure}
Let $\A$ be an \textsc{Abelian} category and $P \cong M \times_T N \in \A$ a fiber product of two objects $M$ and $N$ over a common factor object $T$.
Again we set $A := \ker(P \twoheadrightarrow M)$, $B := \ker(P \twoheadrightarrow N)$, and $S := A + B$.

The four factors
\begin{center}$
  \underbrace{\objXX/\objYY}_{\cong \objXY} \cong (\underbrace{\objXX/\objZZ}_{\cong \objXZ})/(\underbrace{\objYY/\objZZ}_{\cong \objYZ})
$\end{center}
in the first isomorphism theorem applied to $B \leq S \leq P$ can be expressed by four commuting short exact sequences yielding the diagram on the right.

\bigskip
We now formulate the converse of Lemma~\ref{lemm:key} under the assumption that $\Ext^1(T,P)=0$.
\begin{prop} \label{prop:converse_of_key_lemma}
    Under the assumption that $\Ext^1(T,P)=0$ the following two conditions become equivalent:
  \begin{enumerate}
    \item The extension $0 \to A \to N \to T \to 0$ is trivial. \label{prop:converse_of_key_lemma.seq}
    \item $\Ext^1(T,A)=0$. \label{prop:converse_of_key_lemma.ext}
  \end{enumerate}
\end{prop}
\begin{proof}
For the nontrivial implication \eqref{prop:converse_of_key_lemma.seq} $\implies$ \eqref{prop:converse_of_key_lemma.ext} consider the braid diagram below.
Condition~\eqref{prop:converse_of_key_lemma.seq} implies that the connecting homomorphism $\Hom(T,T) \to \Ext^1(T,A)$ is zero, i.e., that $\Ext^1(T,A)$ embeds into $\Ext^1(T,N)$.
The homomorphism $\phi: \Ext^1(T,S) \to \Ext^1(T,N)$ can be written as the composition $\Ext^1(T,S) = \Ext^1(T,A+B) \cong \Ext^1(T,A) + \Ext^1(T,B) \twoheadrightarrow \Ext^1(T,A) \hookrightarrow \Ext^1(T,N)$, showing that the image of $\phi$ is isomorphic to $\Ext^1(T,A)$.
But $\phi$ factors through $\Ext^1(T,P) = 0$ and is hence zero, together with its image $\Ext^1(T,A)$.

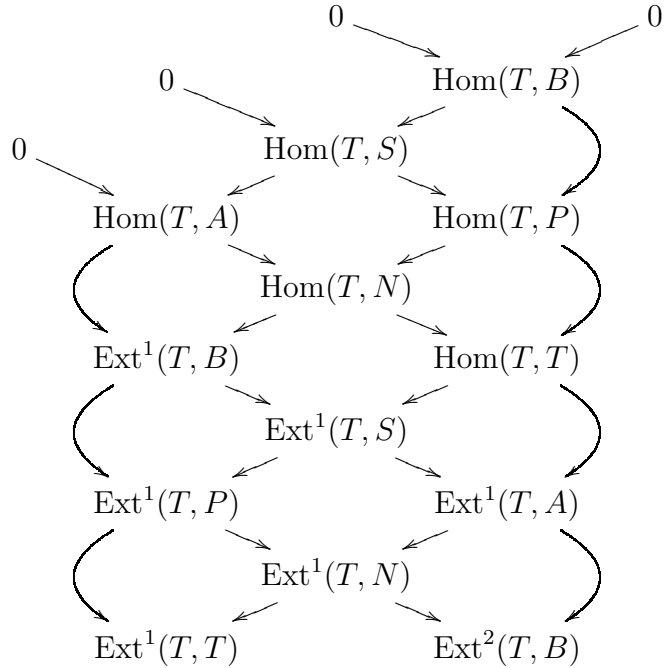
\def\contraobject{T}
\[
  \xymatrix@R=0.25cm@C=0cm{
  &
  &
  0
  \ar[dr]
  &
  &
  0
  \ar[dl]
  \\
  &
  0
  \ar[dr]
  &
  &
  \Hom(\contraobject,\objZZ)
  \ar@/^2.9pc/[dd]
  \ar[dl]
  &
  \mbox{\phantom{Hom(X)}}
  \\
  0
  \ar[dr]
  &
  &
  \Hom(\contraobject,\objYY)
  \ar[dl]
  \ar[dr]
  \\
  &
  \Hom(\contraobject,\objYZ)
  \ar@/_2.9pc/[dd]
  \ar[dr]
  &
  &
  \Hom(\contraobject,\objXX)
  \ar@/^2.9pc/[dd]
  \ar[dl]
  \\
  \mbox{\phantom{Hom(X)}}
  &
  &
  \Hom(\contraobject,\objXZ)
  \ar[dl]
  \ar[dr]
  \\
  &
  \Ext^1(\contraobject,\objZZ)
  \ar[dr]
  \ar@/_2.9pc/[dd]
  &
  &
  \Hom(\contraobject,\objXY)
  \ar[dl]
  \ar@/^2.9pc/[dd]
  \\
  &
  &
  \Ext^1(\contraobject,\objYY)
  \ar[dl]
  \ar[dr]
  \\
  &
  \Ext^1(\contraobject,\objXX)
  \ar[dr]
  \ar@/_2.9pc/[dd]
  &
  &
  \Ext^1(\contraobject,\objYZ)
  \ar[dl]
  \ar@/^2.9pc/[dd]
  \\
  &
  &
  \Ext^1(\contraobject,\objXZ)
  \ar[dl]
  \ar[dr]
  \\
  &
  \Ext^1(\contraobject,\objXY)
  &
  &
  \Ext^2(\contraobject,\objZZ)
  }
\]
\end{proof}

\section*{Acknowledgments}
I would like to thank \textsc{Diego Napp Avelli} for communicating the problem to me and the referees for various improvement suggestions.
I am very grateful to \textsc{Alban Quadrat} who showed us how to do homological algebra in a constructive way.
His influence is everywhere in this work.

\def\cprime{$'$} \def\cprime{$'$} \def\cprime{$'$} \def\cprime{$'$}
\providecommand{\bysame}{\leavevmode\hbox to3em{\hrulefill}\thinspace}
\providecommand{\MR}{\relax\ifhmode\unskip\space\fi MR }
\providecommand{\MRhref}[2]{%
  \href{http://www.ams.org/mathscinet-getitem?mr=#1}{#2}
}
\providecommand{\href}[2]{#2}


\end{document}